\documentclass[letterpaper,11pt]{article}

\newcommand{\conf}[1]{}
\newcommand{\arxiv}[1]{#1}
\usepackage{authblk, natbib, fullpage}
\usepackage{wrapfig}
\usepackage{floatrow}
\usepackage{subcaption}
\usepackage[utf8]{inputenc} 
\usepackage[T1]{fontenc}    
\usepackage{hyperref}       
\usepackage{url}            
\usepackage{booktabs}       
\usepackage{amsfonts}       
\usepackage{nicefrac}       
\usepackage{microtype}      
\usepackage{xcolor}
\usepackage{graphicx}
\usepackage{hyperref}
\hypersetup{
  colorlinks   = true,
  urlcolor     = blue,
  linkcolor    = blue,
  citecolor   = blue
}
\usepackage{amsmath, amsthm}
\usepackage{amsmath}
\usepackage{amsthm}
\usepackage{amssymb}
\usepackage{hyperref} 
\newcommand{\TV}{\text{TV}}
\newcommand{\KL}{\text{KL}}
\newcommand{\HE}{\text{H}}

\newcommand{\CS}{\chi^2}
\newcommand{\TE}{\text{T}}
\newcommand{\EE}{\mathbb{E}}
\newcommand{\Var}{\text{Var}}
\newcommand{\RR}{\mathbb{R}}
\newcommand{\indic}{1}

\newcommand{\ignore}[1]{}

\newcommand{\cH}{\mathcal{H}}
\newcommand{\cX}{\mathcal{X}}
\newcommand{\cP}{\mathcal{P}}
\newcommand{\cO}{\mathcal{O}}

\newcommand{\cT}{\mathcal{T}}
\newcommand{\norm}[1]{\left \lVert #1\right \rVert}

\newtheorem{Theorem}{Theorem}
\newtheorem*{Theorem*}{Theorem}
\newtheorem{Lemma}{Lemma}
\newtheorem{Corollary}{Corollary}

\title{Robust hypothesis testing and distribution estimation in Hellinger distance}
\author{Ananda Theertha Suresh}
\affil{Google Research, New York \\ theertha@google.com}

\begin{document}

\maketitle

\begin{abstract}
We propose a simple  robust hypothesis test that has the same sample complexity as that of the optimal Neyman-Pearson test up to constants, but robust to distribution perturbations under Hellinger distance. We discuss the applicability of such a robust test for estimating distributions in Hellinger distance. We empirically demonstrate the power of the test on canonical distributions.
\end{abstract}
\section{Introduction}
\label{sec:introduction}

\subsection{Simple hypothesis testing}

Hypothesis testing and estimating unknown underlying distributions from samples are fundamental problems in statistics and learning theory respectively. The simplest hypothesis testing scenario is the following. Given two known distributions $P$ and $Q$ over a domain $\cX$ and a set of $n$ independent samples $X^n \triangleq X_1, X_2,\ldots, X_n$ generated from an unknown distribution $R \in \{P, Q\}$, simple hypothesis test asks which of the following two hypotheses is true:
\begin{align*}
\cH_0 &: R = P  \\
\cH_1 &: R = Q.
\end{align*}
The best known hypothesis test is the Neyman-Pearson test, 
which outputs $\cH_0$ if 
\[
\frac{P(X^n)}{Q(X^n)} \geq t,
\]
otherwise outputs $\cH_1$ for a suitable threshold $t$~\citep{neyman1933ix, cover2012elements}. 
There are two types of errors  associated with hypothesis testing: type I error and type II error. Type I error is the probability that the test outputs $\cH_1$ if $\cH_0$ is true 
and  type II error is the probability that the test outputs $\cH_0$ if $\cH_1$ is true. The Neyman-Pearson test achieves the best type II error for a given bound on the type I error.

For simplicity, let the error probability of a hypothesis test be the maximum of type I and type II errors. For a test $T$ and distributions $P$ and  $Q$, let $N^T_\delta(P,Q)$ be the number of samples necessary to achieve error probability $\delta$. Let the optimal sample complexity $N^*_\delta(P,Q)$ be the minimum number of samples necessary to achieve an error probability  $\delta$:
\[
N^*_\delta(P,Q) = \min_{T} N^T_\delta(P,Q).
\]
We need few definitions
to state the optimal sample complexity.
For a function $U: \cX \to \RR^{+}$, the $p$-norm of $U$ is given by 
\[
\norm{U}_p = \left( \int_{x \in \cX} |U(x)|^p dx \right)^{1/p}.
\]
For two distributions $P$ and $Q$ over $\cX$\footnote{We state the results for continuous distributions and the exact results hold for discrete distributions.}, the Hellinger distance between $P$ and $Q$ is given by
\[
\HE(P, Q) = \frac{1}{\sqrt{2}} \norm{\sqrt{P} -  \sqrt{Q}}_2.
\]
The sample complexity of the optimal hypothesis test between $P$ and $Q$ is~\citep{bar2002complexity,canonne2019structure}
\begin{equation}
\label{eq:sim_sam}
N^*_\delta(P, Q) = \Theta \left(\frac{\log ({1}/{\delta})}{\HE^2(P, Q)} \right).
\end{equation}

\subsection{Robust hypothesis testing}

In many natural scenarios, the underlying distribution may not be either of $P$ and $Q$, but close to one of them. This can happen due to a
several reasons such as noisy samples, modelling error, or lack of expressivity
in the class of distributions under consideration.  For example,
suppose we have the following two hypotheses: 
\begin{itemize}
\item $\cH_0$: the number of submissions to a conference every year is $\text{Poi}(5000)$, a Poisson distribution with mean $5000$. 
 \item $\cH_1$: the number of submissions to
the conference every year is $\text{Poi}(6000)$. 
\end{itemize}
It is plausible that in reality,
the number of submissions every year is a Poisson mixture
$(1-\epsilon)\cdot \text{Poi}($5000$) + \epsilon \cdot \text{Poi}(10000)$ for a small $\epsilon$. 
In this scenario, it is desirable for the hypothesis test to overcome the modelling error and output $\cH_0$. It is also preferable for the proposed test to have the same sample complexity as the optimal simple hypothesis test. In this paper, we ask the following question:
\begin{center}
\textit{Is there a test with the same sample complexity as that of the Neyman-Pearson test and is robust to a broad class of distribution perturbations?}
\end{center}
We answer this question affirmatively. To define the broad class of distribution perturbations, we need a measure of closeness between distributions. Since Hellinger distance naturally characterizes the sample complexity of optimal hypothesis testing, we ask if there are robust hypothesis tests under the Hellinger distance.

Given two known distributions $P$
and $Q$ over a domain $\cX$, and a set of $n$ independent samples $X^n$ generated from some distribution $R$,
one can ask which of the following two
hypotheses is true:
\begin{align*}
\cH_0 &: \HE(P, R) < \HE(Q, R)  \\
\cH_1 &: \HE(P, R) > \HE(Q, R).
\end{align*}
For distributions such that $|\HE(P,R) - \HE(Q,R)|$ is arbitrarily small, differentiating between the two hypotheses with finitely many samples would not be possible. Hence we propose $\gamma$-robust hypothesis testing as follows: given two known distributions $P$
and $Q$ over a domain $\cX$, and a set of $n$ independent samples $X^n$ generated from some distribution $R$,
we ask which of the following two
hypotheses is true:
\begin{align*}
\cH_0 &: \gamma \cdot \HE(P, R) \leq \HE(Q, R) \\
\cH_1 &: \HE(P, R) \geq  \gamma \cdot \HE(Q, R),
\end{align*} 
for $\gamma > 1$, where $\gamma$ is the slackness term. If neither of the hypotheses is true, then the test can output either of the hypotheses. As before, we define the error of the test as the maximum of type I and type II errors. For a test
$T$, let $N^T_\delta(P, Q, \gamma)$, be the number of samples necessary to achieve error probability $\delta$ for  $\gamma$-robust hypothesis testing and let $N^*_\delta(P, Q, \gamma)$ be the optimal sample complexity of the $\gamma$-robust hypothesis testing:
\[
N^*_\delta(P, Q, \gamma) = \min_{T} N^T_\delta(P, Q, \gamma).
\]
If a test cannot achieve error probability less than $1/2$  asymptotically, then we say such a test is not $\gamma$-robust. A natural question is to ask if the Neyman-Pearson test is distributionally robust for some $\gamma$. We show that the Neyman-Pearson test is not robust any $\gamma > 1$ by constructing $P,Q$ and a set of distributions $R_m$ such that $\lim_{m \to \infty} \HE(P,R_m) = 0$, but the Neyman-Pearson test outputs $\cH_1$ with high probability. We provide the proof in Section~\ref{app:np_counter}.
\begin{Lemma}
\label{lem:np_counter}
There exists two distributions $P$ and $Q$ such that $N^*_{1/3}(P, Q) = \Theta(1)$ and the
Neyman-Pearson test
is not robust for any $\gamma > 1$.
\end{Lemma}

\subsection{Related works}

We overview robust hypothesis tests with different measures. Let $\cT$ denote the class of all hypothesis tests. For a pair of distributions $P,Q$, and test $T \in \cT$, let $P_e(T, P, Q)$ be the maximum of type I and type II errors of $T$ for distributions $P$ and $Q$. The problem of finding optimal robust hypothesis test can be formulated as
\[
\min_{T \in \cT} \max_{P' \in C(P), Q' \in C(Q)} P_e(T, P', Q'),
\] 
for some convex sets $C(P)$ and $C(Q)$. For tests with $n$ samples, $P_e(T, P', Q')$ is convex in both product spaces $(P')^n$ and $(Q')^n$ over $X^n$. Hence the above min-max problem is convex and the optimal test $T$ can be obtained by computing the least favorable distributions. However this approach can be computationally inefficient.

The first closed form estimator is due to ~\cite{huber1965robust}. They considered a Kolomogorov distance type metric and showed that a clipped log-likelihood test is optimal.  \cite{levy2008robust, gul2017minimax} studied robust distribution hypothesis testing with KL divergence, given by
\[
\KL(P, R) = \int_{x} P(x) \log \frac{P(x)}{R(x)} dx.
\]
However,  KL divergence is not symmetric and simple examples such as the one in Lemma~\ref{lem:np_counter} do not have small KL divergence to the underlying true distributions.

\cite{scheffe1947useful} proposed robust hypothesis test in total variation distance, given by 
\[
\TV(P, Q) = \frac{1}{2} \norm{P - Q}_1.
\]
For any two distributions $P$ and $Q$, Scheffe estimator uses $\cO\left( \frac{\log ({1}/{\delta})}{\TV^2(P,Q)}\right)$ samples
to obtain an error probability of at most $\delta$. It is easy to show that
\begin{equation}
\label{eq:hell-tv}
\frac{1}{2}\TV^2(P,Q) \leq \HE^2(P,Q) \leq \TV(P,Q).
\end{equation}
We provide a simple proof of ~\eqref{eq:hell-tv} in  Section~\ref{app:hell-tv}.
By~\eqref{eq:hell-tv}, the sample complexity of the Scheffe estimator can be worse than the sample complexity of the Neyman-Pearson test. Furthermore, if
the upper bound in~\eqref{eq:hell-tv} is tight, then the sample complexity of the Scheffe estimator can be much higher than the optimal sample complexity as stated in the next lemma. 
\begin{Lemma}
\label{lem:scheffe}
Let $N^S_\delta(P, Q)$ denote the sample complexity of the Scheffe test for simple hypothesis testing.
For any $K > 1$, 
there exists distributions $P,Q$ such that 
\[
N^S_{\delta}(P, Q) = \Omega \left( K \cdot N^*_{\delta}(P, Q) \right).
\]
\end{Lemma}
We relegate the proof of Lemma~\ref{lem:scheffe} in Section~\ref{app:scheffe}.
\section{Contributions}

\subsection{Test statistic}
We propose a test statistic that is distributionally robust for $\gamma > \frac{\sqrt{2}}{\sqrt{2}-1}$ and further has same sample complexity as that of the Neyman-Pearson test up to multiplicative factors. 
 
Since our goal is to come up with a test whose performance guarantee is independent of the underlying domain, it is desirable to have a test of the form $\sum^n_{i=1} f(X_i)$.

Hellinger distance involves a square-root term in its definition and hence finding a $f$ directly is difficult.  Hence, we approximate the Hellinger distance by the symmetric chi-squared statistic, given by
\[
\CS(P,Q) = \norm{\frac{P-Q}{\sqrt{P+Q}}}^2_2 = \int_{x} \frac{(P(x) - Q(x))^2}{P(x) + Q(x)} dx.
\]
The symmetric chi-squared statistic approximates the square of the Hellinger distance to a multiplicative factor of two:
\begin{equation}
\label{eq:chi_hell}
\frac{1}{4} \CS(P, Q) \leq \HE^2(P,Q) \leq \frac{1}{2} \CS(P, Q).
\end{equation}
We provide a derivation of~\eqref{eq:chi_hell} in  Section~\ref{app:chi-hell}. The
symmetric chi-square statistic can be  written as 
\begin{align}
\conf{&} \CS(P,Q) 
\conf{\nonumber \\} &=  \int_{x} \frac{(P(x) - Q(x))^2}{P(x) + Q(x)} dx   \nonumber \\
\arxiv{&=  \int_{x} \frac{P(x)(P(x) - Q(x))}{P(x) + Q(x)} dx -
 \int_{x} \frac{Q(x)(P(x) - Q(x))}{P(x) + Q(x)} dx  \nonumber \\}
& = \EE_{X \sim P} \left[ \frac{P(X) - Q(X)}{P(X) + Q(X)} \right]
- \EE_{X \sim Q} \left[ \frac{P(X) - Q(X)}{P(X) + Q(X)} \right]. 
\label{eq:diff}
\end{align}
~\eqref{eq:diff} motivates the following test statistic. Given $n$ samples $X^n$ from an
unknown distribution $R$, let
\[
\TE(P, Q, X^n) = \frac{1}{n} \sum^n_{i=1} \frac{P(X_i) - Q(X_i)}{P(X_i) + Q(X_i)}.
\] By~\eqref{eq:diff},
\[
\EE_{X^n \sim P}[\TE(P, Q, X^n)] - \EE_{X^n \sim Q}[\TE(P, Q, X^n)] =
\CS(P, Q).
\]
Furthermore,
\[
\EE_{X^n \sim P}[\TE(P, Q, X^n)] + \EE_{X^n \sim Q}[\TE(P, Q, X^n)] = 0.
\]
Hence,
\begin{align*}
\EE_{X^n \sim P}[\TE(P, Q, X^n)] 
& = - \EE_{X^n \sim Q}[\TE(P, Q, X^n)] \conf{\\&} = \frac{1}{2} \CS(P, Q).
\end{align*}
Hence, a natural test is to output $P$ if 
$
\TE(P, Q, X^n) > 0
$
and $Q$ if $\TE(P, Q, X^n) < 0$, while breaking ties randomly. 
We refer to this test as \textsc{HellingerTest}.
 \textsc{HellingerTest} does not have any tunable hyperparameters and just depends on the underlying distributions $P$ and $Q$. Instead of comparing to zero, one can compare to a threshold $t$ to get precise trade offs between type I and type II errors.

\subsection{Theoretical guarantees}
We show that 
\textsc{HellingerTest} is also robust to distribution perturbations in Hellinger distance and has the optimal sample complexity of simple hypothesis testing.
Thus \textsc{HellingerTest} guarantees robustness in Hellinger distance for free.
\begin{Theorem}
\label{thm:main}
Let $N^{H}_\delta(P,Q,\gamma)$ be the sample complexity of the \textsc{HellingerTest}.
For $\gamma > \frac{\sqrt{2}}{\sqrt{2}-1}$ and any $\delta$,
\[
N^H_\delta(P, Q, \gamma) = \frac{c_\gamma \log ({1}/{\delta})}{\HE^2(P,Q)} = \Theta\left( N^*_\delta(P, Q)\right),
\]
where $c_\gamma$ is a constant that depends on $\gamma$.
\end{Theorem}
We also show a lower-bound on the performance of the \textsc{HellingerTest}. 
\begin{Theorem}
\label{thm:zero_mean}
For every $\gamma < \frac{1}{\sqrt{2}-1}$, there exists  distributions $P$, $Q$, and $R$
such that $\frac{\HE(Q,R)}{\HE(P,R)} \geq \gamma$ and
\[
\EE_{X^n \sim R}[T(P, Q, X^n)] = 0.
\]
\end{Theorem}
\begin{figure}[t]
\caption{Illustration for Bernoulli distributions $P = B(1/4)$ and $Q= B(3/4)$. The black line plots $\HE(P,R) = \HE(Q,R)$. \textsc{HellingerTest} is robust to changes in distributions in the blue region and not robust in the orange region.}
\label{fig:hellinger}
{\includegraphics[scale=0.25]{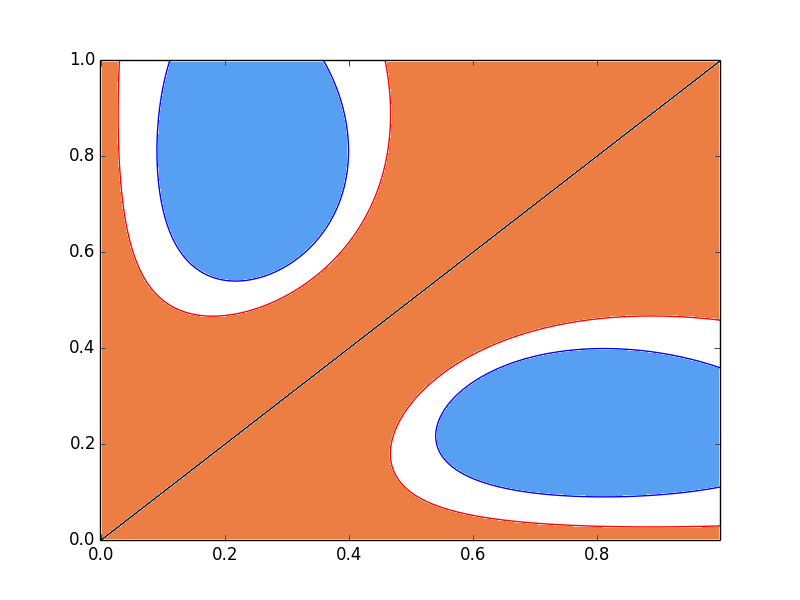}}
\end{figure} 
 The results are illustrated in Figure~\ref{fig:hellinger} for Bernoulli distributions.
Bridging the $\sqrt{2}$-gap between constants in Theorems~\ref{thm:main} and ~\ref{thm:zero_mean} is an interesting future direction.

\begin{figure*}[t]
\centering
  \caption{Comparision of Neyman-Pearson and \textsc{HellingerTest} for different distributions.}
\label{fig:normal}
 \begin{subfigure}[t]{.3\linewidth}
    \centering\includegraphics[width=1.0\linewidth]{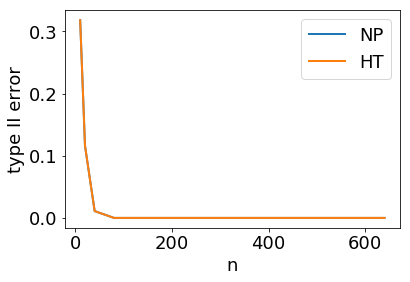}
    \caption{$P=B(0)$, $Q=B(0.1)$.}
  \end{subfigure}
  \begin{subfigure}[t]{.3\linewidth}
    \centering\includegraphics[width=1.0\linewidth]{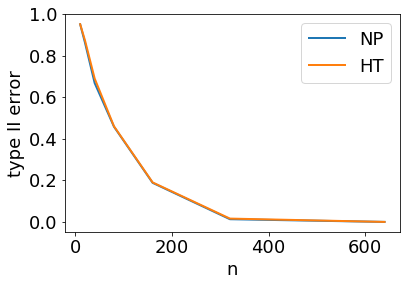}
    \caption{$P=B(0.5)$, $Q=B(0.6)$.}
  \end{subfigure}
  \begin{subfigure}[t]{.34\linewidth}
    \centering\includegraphics[width=0.9\linewidth]{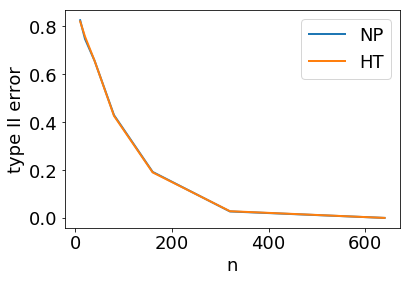}
    \caption{$P=N(0, 1)$, $Q=N(0.2, 1)$.}
  \end{subfigure}
\end{figure*}

\subsection{Implications for distribution estimation}

Distribution robust hypothesis testing can be rewritten as a test that given pair of distributions $P,Q$ and samples from an unknown distribution $R$, finds a distribution $\hat{R} \in \{P, Q\}$ such that $\HE(R, \hat{R}) \leq \gamma \cdot \min(\HE(P, R), \HE(Q, R))$. 

Such a distribution robust hypothesis testing can be used as a subroutine for learning distributions. Consider the following learning problem: given $n$ samples $X^n \sim P \in \cP$,  find an estimate $\hat{P}$ such that $\HE(P, \hat{P}) \leq \epsilon$. 

A natural algorithm is to obtain an $\epsilon$-cover of $\cP$, denoted by $\cP_\epsilon$ and run the robust hypothesis test between every pair of distributions and output the distribution that wins in the maximum number of tests~\citep{devroye2012combinatorial}. It can be shown that the overall algorithm selects a distribution that is at most $c \cdot \gamma\cdot \epsilon$ away from the true distribution, where $c$ is a constant. We refer readers to \citep[Section 6.8]{devroye2012combinatorial} for a detailed description of this algorithm. The above algorithm can be further modified 
with improved run time~\citep{acharya2014sorting, acharya2018maximum} and to provide differential privacy~\citep{bun2019private}.

Perhaps the most popular estimator is the Scheffe test, which studied the problem under the total variation distance~\citep{scheffe1947useful, yatracos1985rates}. Scheffe test has been used in variety of works including learning Gaussian mixtures~\citep{daskalakis2014faster, suresh2014near, ashtiani2018nearly}, $k$-modal distributions~\citep{daskalakis2012learning}, log-concave distributions~\citep{diakonikolas2016learning}, and piece-wise polynimial distributions~\citep{chan2014efficient}. 

The proposed test can be used in place of the Scheffe's estimator in the above papers to obtain learning guarantees in the Hellinger distance.

\subsection{Modifications for differential privacy}

Differential privacy has become the standardized notion of privacy in statistics. We refer readers to \citep{CynthiaA14} for details on differential privacy. Optimal hypothesis test with differential privacy was proposed by~\cite{canonne2019structure}. 
Let 
\[
\Delta(P, Q) = \max_{x \in \cX}  \frac{|P(x) - Q(x)|}{P(x) + Q(x)}  \leq 1.
\]
Changing one sample changes the proposed test statistic $T(P, Q, X^n)$ by at most
$2\Delta(P, Q)/n$, hence \textsc{HellingerTest} can be modified to a $\epsilon$-differentially private test by adding Laplace noise,
\[
T_\epsilon (P, Q, X^n) = T(P, Q, X^n) + \frac{Z}{n},
\]
where $Z$ is a Laplace random variable with parameter $2\Delta(P, Q)/\epsilon$. While this algorithm is not optimal in general, as we show below, it is optimal for $\epsilon > 1$ and further has the advantage that it is parameter-free and simple to use. 

\begin{Corollary}
\label{cor:dp}
$T_\epsilon (P, Q, X^n)$ is an $\epsilon$-DP algorithm. Furthermore, its sample complexity is optimal and is same as that of the non-private  complexity up to constants for 
\[
\epsilon \geq \max_{x \in \cX}  \frac{|P(x) - Q(x)|}{P(x) + Q(x)}.
\]
If $\Delta(P,Q)$ is unknown, instead of adding Laplace noise with parameter $2\Delta(P, Q)/n$, one can add Laplace noise with parameter $2/n$ and it would be near-optimal for  $\epsilon \geq 1$ for all $P,Q$.
\end{Corollary}
We relegate the proof to Section~\ref{app:dp}.
The above algorithm is amenable to the same clipping strategy proposed by \cite{canonne2019structure}
and can be modified to obtain the optimal sample complexity with differential privacy.

\subsection{Implications for other measures}
If $\HE(P,R) \leq \frac{1}{\gamma + 1} \HE(P,Q)$ ,
then by the triangle inequality,
\begin{align*}
\HE(Q,R) \conf{&} \geq \HE(P, Q) - \HE(P, R) \conf{\\&} \geq (\gamma + 1 - 1) \HE(P,R) \conf{\\&} =  \gamma \HE(P,R),
\end{align*}
Similarly, if $\HE(Q,R) \leq \frac{1}{\gamma + 1} \HE(P,Q)$,
then $\HE(P,R) \geq \gamma \HE(Q,R)$. Hence, if there is $\gamma$-robust hypothesis test, it can also differentiate between the following two hypotheses:
\begin{align*}
    \cH_0  : \HE(P, R) \leq \frac{1}{\gamma + 1} \HE(P,Q) 
    \\ 
    \cH_1  : \HE(Q, R) \leq \frac{1}{\gamma + 1} \HE(P,Q),
\end{align*}
such that the sample complexity is same as that of the Neyman-Pearson test. Furthermore if there is a measure $d$ such that Hellinger distance is upper bounded by some function of $d$, then the test works for even that class of distributions. This observation yields the following corollary.
\begin{Corollary}
\label{cor:diff_measures}
Let $\beta < \frac{\sqrt{2}-1}{2\sqrt{2}-1}$.
\textsc{HellingerTest} has the same complexity as the optimal simple hypothesis testing for the following composite hypothesis testing scenarios:
\begin{enumerate}
\item Hellinger distance:
\begin{align*}
    \cH_0  : \HE(P, R) \leq \beta \HE(P,Q) 
\\  \cH_1  : \HE(Q, R) \leq \beta \HE(P,Q).
\end{align*}
\item Total variation distance:
\begin{align*}
    \cH_0 : \TV(P, R) \leq \beta^2 \HE^2(P,Q) \\ 
    \cH_1 : \TV(Q, R) \leq \beta^2 \HE^2(P,Q).
\end{align*}
\item KL distance $\KL(\cdot, R)$:
\begin{align*}
    \cH_0  : \KL(P, R) \leq 2\beta^2\HE^2(P,Q)   \\  
    \cH_1 : \KL(Q, R) \leq 2\beta^2\HE^2(P,Q).
\end{align*}
\item KL distance $\KL(R, \cdot)$:
\begin{align*}
    \cH_0  : \KL(R, P) \leq 2\beta^2\HE^2(P,Q)   \\  
    \cH_1 : \KL(R, Q) \leq 2\beta^2\HE^2(P,Q).
\end{align*}

\end{enumerate}
\end{Corollary}

\section{Experiments}

We first evaluate Neyman-Pearson test and \textsc{HellingerTest} on few canonical distributions without distribution perturbations and demonstrate that they have similar performance. 
For these experiments, we set the threshold $t$ such that the type I error is at most $0.05$. The results are in Figure~\ref{fig:normal}. The experiments are averaged over $1000$ trials for statistical consistency. The behavior of Neyman-Pearson and \textsc{HellingerTest} are similar. 

We then evaluate the effect of robustness for Gaussian distributions and Bernoulli distributions in Figures~\ref{fig:robust_gauss} and~\ref{fig:robust_bern} respectively.
The experiments demonstrate that  \textsc{HellingerTest}  is robust to distribution perturbations, where as the Neyman-Pearson test is not.

\begin{figure}[t]
\centering
  \caption{Comparision of Neyman-Pearson and \textsc{HellingerTest} with perturbed Gaussian distributions. $P = N(0,1)$, $Q=N(0.2, 1)$, and $R = (1-w) P + w N(100, 1)$.}
\label{fig:robust_gauss}
\centering
 \begin{subfigure}[t]{.47\linewidth}
    \centering\includegraphics[width=1.0\linewidth]{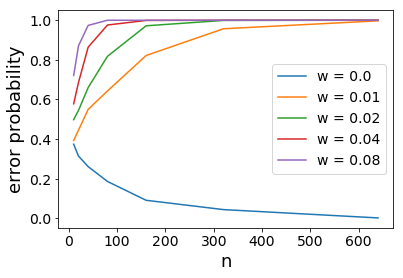}
    \caption{Neyman-Pearson test.}
  \end{subfigure}
  \begin{subfigure}[t]{.47\linewidth}
    \centering\includegraphics[width=1.0\linewidth]{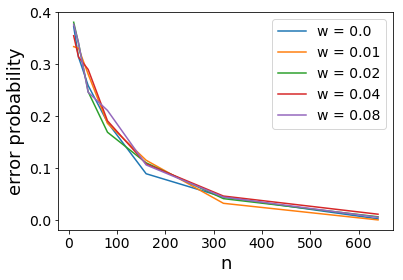}
    \caption{ \textsc{HellingerTest}.}
  \end{subfigure}
\end{figure}
\begin{figure}
\centering
  \caption{Comparision of Neyman-Pearson and \textsc{HellingerTest} with perturbed Bernoulli distributions. $P = B(0)$, $Q=B(0.1)$, and $R = B(r)$.}
\label{fig:robust_bern}
\centering
 \begin{subfigure}[t]{.47\linewidth}
    \centering\includegraphics[width=1.0\linewidth]{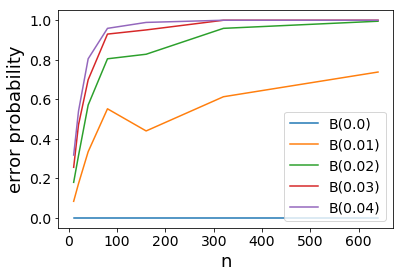}
    \caption{Neyman-Pearson test.}
  \end{subfigure}
  \begin{subfigure}[t]{.47\linewidth}
    \centering\includegraphics[width=1.0\linewidth]{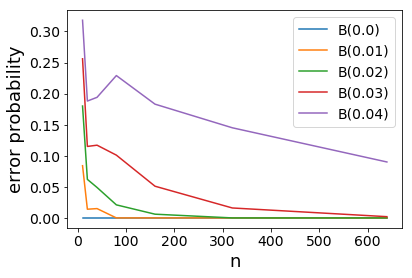}
    \caption{ \textsc{HellingerTest}.}
  \end{subfigure}
\end{figure}
\section{Proof of Theorem~\ref{thm:main}}

The analysis of the test statistic involves computing the variance and the expectation and using the Bernstein inequality. The next lemma bounds the variance in terms of Hellinger distance. 
\begin{Lemma}
\label{lem:var}
For any two distributions $P$ and $Q$, if $X^n \sim R$, then 
\[
\Var(\TE(P, Q, X^n)) \leq \frac{55}{n} \max \left(\HE^2(P, R), \HE^2(Q,R) \right).
\]
\end{Lemma}
\begin{proof}
Since $X_1, X_2, \ldots, X_n$, are i.i.d. samples from $R$, 
\begin{align*}
\Var(\TE(P, Q, X^n)) 
&= \frac{1}{n} \Var \left(\frac{P(X_1) - Q(X_1)}{P(X_1) + Q(X_1)} \right) 
\conf{\\}& \leq \frac{1}{n} \EE_{X \sim R} \left[\left( \frac{P(X) - Q(X)}{P(X) + Q(X)}\right)^2 \right].
\end{align*}
For $\beta > 1$, let $S$ be the set given by $\{x :
R(X) > \beta (P(X) + Q(X))\}$. For $x \in S$,
\begin{equation}
\label{eq:var_temp_1}
(\sqrt{P}(x) - \sqrt{R}(x))^2 \geq R(x)(\sqrt{\beta} -1 )^2/\beta.
\end{equation}
Hence,
\begin{align*}
& \EE_{X \sim R} \left[\left( \frac{P(X) - Q(X)}{P(X) + Q(X)}\right)^2 \right]  \\
&=  \EE_{X \sim R} \left[\left( \frac{P(X) - Q(X)}{P(X) + Q(X)}\right)^2 \indic_{x \in S} \right]
\conf{\\&} + \EE_{X \sim R} \left[\left( \frac{P(X) - Q(X)}{P(X) + Q(X)}\right)^2 \indic_{x \notin S} \right] \\
& \stackrel{(a)}{\leq} \EE_{X \sim R}[\indic_{x \in S}] + \beta \CS(P, Q) \\
& \stackrel{(b)}{\leq} \frac{2\beta}{(1-\sqrt{\beta})^2} 
\HE^2(P, R) + \beta \CS(P, Q)\\
& \stackrel{(c)}{\leq} \frac{2\beta}{(1-\sqrt{\beta})^2}  \HE^2(P, R) + 4\beta \HE^2(P, Q) \\
&  \stackrel{(d)}{\leq} \frac{2\beta}{(1-\sqrt{\beta})^2}  \HE^2(P, R) + 16\beta \HE^2(P, Q) \\
& \leq \left( \frac{2\beta}{(1-\sqrt{\beta})^2}  + 16 \beta \right) \cdot \max \left(\HE^2(P, R), \HE^2(Q,R) \right).
\end{align*}
$(a)$ follows from the definition of set $S$
 and $\CS$ statistic.~\eqref{eq:var_temp_1} implies $(b)$.~\eqref{eq:chi_hell} implies $(c)$.
$(d)$ follows from triangle inequality and the fact that $(a+b)^2 \leq 2a^2 + 2b^2$. Minimizing over $\beta > 1$ yields the lemma.
\end{proof}
In the next lemma we bound the expectation, which is the  crucial part of our proof. 
\begin{Lemma}
\label{lem:mean}
For distributions $P, Q$, and $R$, if 
$
\HE(Q, R) \geq \frac{\sqrt{2}}{\sqrt{2}\alpha-1} \HE(P, R),
$
for $\alpha \in (1/\sqrt{2}, 1)$, then
\[
\EE[\TE(P, Q, X^n)] \geq 2(1-\alpha^2) \HE^2(Q,R).
\]
\end{Lemma}
\begin{proof}
Since $X_1, X_2, \ldots, X_n$ are i.i.d. samples from $R$,
\begin{align*}
    \conf{&} 2 \EE[\TE(P, Q, X^n)] \conf{\\}
\arxiv{&= \int_{x \in \cX} \frac{R(x)(P(x) - Q(x))}{P(x)+Q(x)} dx \\
& =  \int_{x \in \cX} \frac{(Q(x)-R(x))^2 - (P(x)-R(x))^2 + P^2(x) - Q^2(x)}{P(x)+Q(x)} dx \\}
& = \int_{x \in \cX} \frac{(Q(x)-R(x))^2 - (P(x)-R(x))^2}{P(x)+Q(x))} \conf{\\&}+ (P(x) - Q(x) dx \\
& =  \int_{x \in \cX} \frac{(Q(x)-R(x))^2 - (P(x)-R(x))^2}{P(x)+Q(x)} dx,
\end{align*}
where the last equality follows from the fact that $P$ and $Q$ are probability distributions 
and hence integrates to $1$. 
For any three non-negative numbers $p, q$, and $r$, \conf{it can be shown that}
\begin{align*}
& \frac{(q-r)^2}{p+q} - \frac{(p-r)^2}{p+q} \\
\arxiv{& = \frac{(q-r)^2}{q+r} - \frac{(p-r)^2}{p+r} 
+ \frac{(q-r)^2}{p+q} - \frac{(q-r)^2}{q+r} + \frac{(p-r)^2}{p+r} - \frac{(p-r)^2}{p+q} \\
& =  \frac{(q-r)^2}{q+r} - \frac{(p-r)^2}{p+r} 
+ \frac{(q-r)^2(r-p)}{(p+q)(q+r)} + \frac{(p-r)^2(q-r)}{(p+q)(p+r)} \\
& =  \frac{(q-r)^2}{q+r} - \frac{(p-r)^2}{p+r} 
+ \frac{(q-r)(p-r)}{(p+q)(q+r)(p+r)}\left(-(q-r)(p+r) + (p-r)(q+r) \right)\\}
& =  \frac{(q-r)^2}{q+r} - \frac{(p-r)^2}{p+r} 
+ \frac{2r(q-r)(p-r)(p-q)}{(p+q)(q+r)(p+r)}.
\end{align*}
Applying the above equality in the expectation, and
substituting the definition of $\CS$ statistic,
\begin{align*}
2 \EE[\TE(P, Q, X^n)] 
\conf{}& \geq \CS(Q, R) - \CS(P, R) \conf{\\&} -
2\norm{\frac{(P-Q)(P-R)(Q-R)R}{(P+Q)(P+R)(Q+R)}}_1
\end{align*}
We first bound the last term. 
\begin{align*}
\conf{&}    \norm{\frac{(P-Q)(P-R)(Q-R)R}{(P+Q)(P+R)(Q+R)}}_1 \conf{\\}
    & \leq\norm{\frac{(P-R)(Q-R)R}{(P+R)(Q+R)}}_1 \\
    & \leq \norm{\frac{(P-R)(Q-R)}{\sqrt{(P+R)(Q+R)}}}_1 \norm{\frac{R}{\sqrt{(P+R)(Q+R)}}}_\infty \\
 & \leq \norm{\frac{(P-R)(Q-R)}{\sqrt{(P+R)(Q+R)}}}_1 \\
 & \leq \norm{\frac{(P-R)}{\sqrt{(P+R)}}}_2
\norm{\frac{(Q-R)}{\sqrt{(Q+R)}}}_2 \\
& = \sqrt{\CS(Q, R) \CS(P, R)},
\end{align*}
where the last inequality follows by the Cauchy-Schwarz inequality. Combining the above equations, 
\begin{align*}
2 \EE[\TE(P, Q, X^n)] 
& \geq \CS(Q, R) - \CS(P, R) \conf{\\&}- 2  \sqrt{\CS(Q, R) \CS(P, R)}.
\end{align*}
Let $\gamma = \frac{\HE(Q,R)}{\HE(P,R)} \geq 2$.
We now lower bound the above term in terms of Hellinger distances. 
\begin{align*}
& \CS(Q, R) - \CS(P, R) - 2  \sqrt{\CS(Q, R) \CS(P, R)} \\
& \stackrel{(a)}{\geq}  \CS(Q, R) - 4\HE^2(P, R) - 4  \sqrt{\CS(Q, R) \HE^2(P, R)} \\
& \stackrel{(b)}{\geq}   2 \HE^2(Q,R) - 4\HE^2(P, R) - 4  \sqrt{2 \HE^2(Q,R) \HE^2(P, R)} \\
& \stackrel{(c)}{\geq} 2 \HE^2(Q,R) - \frac{4\HE^2(Q, R)}{\gamma^2} - \frac{4\sqrt{2} \HE^2(Q,R)}{\gamma},
\end{align*}
where $(a)$ follows by~\eqref{eq:chi_hell}.
 $z - 4\HE^2(P, R) - 4\sqrt{z \HE^2(P, R)} $ is an increasing function of $z \in [4\HE^2(P,R), \infty)$. Furthermore by~\eqref{eq:chi_hell}, $\CS(Q, R)  \geq 2 \HE^2(Q,R) \geq 2 \gamma^2 \HE^2(P,R) \geq 4 \HE^2(P,R)$. Hence substituting a lower bound on $\CS(Q, R)$ yields $(b)$. $(c)$ follows from the definition of $\gamma$.
Hence,
\begin{align*}
\EE[\TE(P, Q, X^n)] & \geq \HE^2(Q,R) 
\left(1 - \frac{2}{\gamma^2} 
- \frac{2\sqrt{2}}{\gamma}\right)  \\
& = \HE^2(Q,R) 
\left(2 - \left(\frac{\sqrt{2}}{\gamma} + 1 \right)^2\right).
\end{align*}
Substituting $\gamma = \frac{\sqrt{2}}{\sqrt{2}\alpha -1}$ yields the result.
\end{proof}
The proof of Theorem~\ref{thm:main} uses the Bernstein inequality, which we state for completeness.
\begin{Lemma}[Bernstein inequality]
Let $Z_1, Z_2, \ldots Z_n$ are i.i.d. random variables and $M = \max_{Z} |Z|$ and $\sigma(Z)$ denote its variance. Then
with probability at least $ 1- \delta$,
\[
 \EE[Z] - \frac{1}{n} \sum_i Z_i 
\leq  2\sigma(Z)\sqrt{\frac{\log \frac{1}{\delta}}{n}} + \frac{4M}{3n} \log \frac{1}{\delta}.
\]
\end{Lemma}
\begin{proof}[Proof of Theorem~\ref{thm:main}]
Without loss of generality, we assume $\HE(Q,R) \geq \gamma \HE(P, R)$. 
Let $\gamma = \frac{\sqrt{2}}{\sqrt{2}\alpha - 1}$ for $\alpha  \in (1/\sqrt{2}, 1)$. We apply Bernstein theorem based on our bounds on expectations and variances. 
In particular, let $Z = \frac{P(x) - Q(x)}{P(x)+Q(x)}$. Hence by Lemma~\ref{lem:mean},
\[
\EE[Z] \geq 2(1-\alpha^2) \HE^2(Q,R).
\]
By Lemma~\ref{lem:var}, 
\[
\sigma(Z) \leq 8 \HE(Q,R),
\]
and $M = \max_{Z} |Z| \leq 1$. Hence, 
with probability at least $1-\delta$,
\begin{align}
\TE(P, Q, X^n) & \leq 2(1-\alpha^2) \HE^2(Q,R) \conf{\nonumber \\&}- \frac{c \HE(Q,R)\sqrt{\log \frac{1}{\delta}}}{\sqrt{n}} - \frac{c \log \frac{1}{\delta}}{n},
\label{eq:proof}
\end{align}
for some constant $c > 1$. Hence if $n \geq \frac{100c^2 \log \frac{1}{\delta}}{\HE^2(Q,R)(1-\alpha^2)^2}$, then with probability at least $1- \delta$,
\[
\TE(P, Q, X^n) \geq (1-\alpha^2) \HE^2(Q,R) > 0.
\]
The theorem follows by observing that
\[
\HE(P,Q) \leq \HE(P,R) + \HE(Q,R)
\leq (\gamma + 1) \HE(Q,R).
\]
\end{proof}
\section{Proofs of other results}

\subsection{Proof of Lemma~\ref{lem:np_counter}}
\label{app:np_counter}

We give a simple example with Bernoulli distributions. Similar results
hold for other distributions such as Gaussian mixtures. Let $B(p)$ be the Bernoulli distribution with parameter $p$.  Let $P=B(0)$
and $Q=B(1/2)$. By~\eqref{eq:sim_sam}, $N^*_{1/3}(P,Q) = \Theta(1)$.  

Let $R = B(1/(16\gamma^2))$. It can be shown that $\HE(P, R) \leq 1/(4\gamma)$ and  $\HE(Q, R) \geq 1/3$ for $\gamma > 1$. Hence,
\[
\frac{\HE(Q,R)}{\HE(P, R)} \geq \frac{\frac{1}{3}}{\frac{1}{4\gamma}} \geq \gamma.
\]
Let $\delta > 0$. Given $n \geq 16\gamma^2 \log \frac{1}{\delta} $ samples from $R$,
then with probability at least $ 1-\delta$, at least one of the symbols is $1$. Then, $P(X^n) = 0$ and $Q(X^n) = 1/2^n$ and for any finite threshold $t$, the test outputs $\cH_1$. Hence, the error probability of the Neyman-Pearson test is at least $1-\delta$. Taking the limit as $\delta\to 0$ shows that Neyman-Pearson test is not robust.

\subsection{Proof of Lemma~\ref{lem:scheffe}}
\label{app:scheffe}

Let $\cX = \{0, 1, 2\}$ and $\epsilon = 1/(2K)$. Let $P$ be given by
$P(0) = 1/2$, $P(1) = 1/2-\epsilon$, $P(2) = \epsilon$. Let $Q$ be given by
$Q(0) = 1/2-\epsilon$, $Q(1) = 1/2 + \epsilon$, and $Q(2) = 0$. 

The Hellinger distance between $P$ and $Q$ is $\Theta(\sqrt{\epsilon})$. By~\eqref{eq:sim_sam}, $N^*_{\delta}(P, Q) = \Theta(\log(1/\delta)/\epsilon)$.

Scheffe's test
measures empirical probability of $S = \{x : P(x) \geq Q(x)\}$ and infers the underlying hypothesis. For the above example, $S = \{0, 2\}$. For this set $S$, $P(S) = \frac{1}{2} + \epsilon$ and $Q(S) = \frac{1}{2} - \epsilon$. Hence, the sample complexity of Scheffe test is lower bounded by the sample complexity of the best hypothesis  test between $B(1/2+\epsilon)$ and $B(1/2-\epsilon)$. Therefore by~\eqref{eq:sim_sam},  
\begin{align*}
N^S_{\delta}(P, Q) &= \Omega(\log(1/\delta)/\epsilon^2) \conf{\\&}= \Omega(N^*_{\delta}(P, Q)/\epsilon)\conf{\\&} = \Omega(K \cdot N^*_{\delta}(P, Q)).
\end{align*}

\subsection{Proof of Theorem~\ref{thm:zero_mean}}
\label{app:zero_mean}
Let $Q = B(0)$, $P=B(2\epsilon)$, and $R=B(\epsilon)$, where we choose $\epsilon$ later. For this choice of $P, Q$, and $ R$, 
\[
\EE_{X^n \sim R}[T(P, Q, X^n)] = 0.
\]
We now bound the ratio of Hellinger distances,
\begin{align*}
    \frac{\HE^2(Q,R)}{\HE^2(P,R)}
      & = \frac{1 - \sqrt{1-\epsilon}}{1 - \sqrt{(1-2\epsilon)(1-\epsilon)} - \sqrt{2}\epsilon}.
\end{align*}
Taking the right limit as $\epsilon \to 0$ and using L'Hôpital's rule yields,
\begin{align*}
 \lim_{\epsilon \to 0^+}    \frac{1 - \sqrt{1-\epsilon}}{1 - \sqrt{(1-2\epsilon)(1-\epsilon)} - \sqrt{2}\epsilon}
\conf{} = \frac{1/2}{3/2 - \sqrt{2}} 
\conf{\\} = \frac{1}{3-2\sqrt{2}} 
\conf{} = \frac{1}{(\sqrt{2} - 1)^2}.
\end{align*}
Hence, for every $\gamma < \frac{1}{\sqrt{2} -1}$, 
there exists an $\epsilon$ such that $\frac{\HE^2(Q,R)}{\HE^2(P,R)} \geq \gamma^2$.

\subsection{Proof of Corollary~\ref{cor:dp}}
\label{app:dp}
We provide the proof when $R=P$. The proof for the case when $R=Q$ is similar and omitted. 
By the tail bounds of the Laplace random variable, there exists a constant $c'$ such that with probability at least $ 1- \delta/2$,
\begin{align*}
T_\epsilon(P, Q, X^n) &\geq T(P, Q, X^n)- \frac{2c' \Delta(P,Q) \cdot \log \frac{2}{\delta}}{n \epsilon} \\
& \geq  T(P, Q, X^n)- \frac{c' \cdot \log \frac{2}{\delta}}{n}.
\end{align*}
Since $R=P$, by~\eqref{eq:chi_hell},
\[
\EE[\TE(P, Q, X^n)] = \frac{1}{2}\CS(P,Q) \geq  \HE^2(P,Q).
\]
Similar to the proof of Theorem~\ref{thm:main}, applying the Bernstein inequality yields that with probability at least $ 1- \delta/2$,
\[
\TE(P, Q, X^n) \geq \HE^2(P,Q) - \frac{c \HE(Q,P)\sqrt{\log \frac{2}{\delta}}}{\sqrt{n}} - \frac{c \log \frac{2}{\delta}}{n}.
\]
Combining the above two equations yields that with probability at least $1-\delta$,
\begin{align*}
T_\epsilon(P, Q, X^n) & \geq \HE^2(P,Q) - \frac{c \HE(Q,P)\sqrt{\log \frac{2}{\delta}}}{\sqrt{n}} \conf{\\&}- \frac{(c + c') \log \frac{2}{\delta}}{n}.
\end{align*}
Hence if $n \geq c'' \left( \frac{\log \frac{1}{\delta}}{\HE^2(P,Q)}\right)$ for a sufficiently large constant $c''$, then with probability at least $1-\delta$,
\[
T_\epsilon(P, Q, X^n) > 0,
\]
and hence the result.

\section{Relationship between distances}

\subsection{Relationship between Hellinger distance and total variation distance}
\label{app:hell-tv}
\textbf{Upper bound:}
\begin{align*}
 \HE^2(P,Q) & = \frac{1}{2} \left \| \sqrt{P} - \sqrt{Q} \right \|^2_2  \\
 &\leq \frac{1}{2} \left \| (\sqrt{P} - \sqrt{Q})
  (\sqrt{P} + \sqrt{Q})\right \|_1  \\
   & =  \frac{1}{2} \left \|P - Q\right \|_1  \\
   & =  \TV(P, Q).
\end{align*}
\textbf{Lower bound:} 
\begin{align*}
 \HE^2(P,Q) & = \frac{1}{2} \left \| \sqrt{P} - \sqrt{Q} \right \|^2_2  \\
 & \stackrel{(a)}{\geq} \frac{1}{8} \left \| \sqrt{P} - \sqrt{Q} \right \|^2_2 \cdot  \left \| \sqrt{P} + \sqrt{Q} \right \|^2_2  \\
  & \stackrel{(b)}{\geq} \frac{1}{8} \left \| P - Q\right \|^2_1   \\
   & =  \frac{1}{2} \TV^2(P, Q),
\end{align*}
where $(a)$ follows from the fact that $\left \| \sqrt{P} + \sqrt{Q} \right \|_2 \leq 2$ and $(b)$ uses the Cauchy-Schwarz inequality.
\subsection{Relationship between Hellinger distance and symmetric chi-squared statistic}
\label{app:chi-hell}
\begin{align*}
 \HE^2(P,Q) & = \frac{1}{2} \left \| \sqrt{P} - \sqrt{Q} \right \|^2_2  \\
 & = \frac{1}{2}\left  \| \frac{(\sqrt{P} - \sqrt{Q})(\sqrt{P} + \sqrt{Q})}{(\sqrt{P} + \sqrt{Q})} \right \|^2 _2 \\
  & = \frac{1}{2} \left \| \frac{P - Q}{(\sqrt{P} + \sqrt{Q})} \right \|^2_2.
\end{align*}
The proof of~\eqref{eq:chi_hell} follows by observing that for every $x$,
\[
\sqrt{P(x) + Q(x)} \leq \sqrt{P(x)} + \sqrt{Q(x)} \leq \sqrt{2(P(x) + Q(x))}.
\]

\section{Conclusion}

We proposed a simple robust hypothesis test that has the same complexity of the optimal Neyman-Pearson test up to constants and is robust to distribution perturbations in Hellinger distance. The test is relatively parameter free and easy to use. We evaluated the test on synthetic distributions and also provided extensions with differential privacy. Bridging the $\sqrt{2}$-gap between the upper and lower bounds is an interesting future direction.

\newpage
\bibliography{arxiv}
\bibliographystyle{abbrvnat}

\end{document}